\documentclass[11pt]{article}
\usepackage[numbers]{natbib}
\usepackage[hidelinks]{hyperref}
\usepackage{color}
\usepackage{amsmath}
\usepackage{amsfonts}
\usepackage{amsthm}
\usepackage[bb=ams,scr=euler]{mathalpha}
\usepackage{graphicx}
\usepackage{tikz}
\usepackage{tikz-cd}
\usepackage{graphics}
\usepackage{pinlabel}

\author{Habib Alizadeh}
\date{} 

\newtheorem{theorem}{Theorem}[section]

\newtheorem{lemma}[theorem]{Lemma}
\newtheorem{question}[theorem]{Question}

\newtheorem{remark}[theorem]{Remark}
\newtheorem{definition}[theorem]{Definition}
\newtheorem{def/theorem}[theorem]{Theorem/Definition}

\newcommand{\norm}[1]{\left\lVert#1\right\rVert}

\begin{document}

\begin{center}
\huge{Hamiltonian fragmentation in dimension four with application to spectral estimators} \\
\vspace{0.5cm}
\large{Habib Alizadeh}
\end{center}
\vspace{1cm}

\begin{abstract}
\sloppy We prove a new Hamiltonian extension and consequently a fragmentation result in dimension $4$ for the symplectic manifold $\mathbb{D}^{2}\times \mathbb{D}^{2}$. Polterovich and Shelukhin have recently constructed a family of functionals on the space of time dependent Hamiltonian functions on $S^{2}(1) \times S^{2}(a)$ for certain rational $0 < a < 1$, called Lagrangian spectral estimators. Using our fragmentation result we prove that the restriction of their functionals to the subdomain $\mathbb{D}^{2}(c) \times \mathbb{D}^{2}(a)$ is a uniformly $C^{0}$-continuous functional where $0 < c < 1$. As an application of our results, we show that the complement of a Hofer ball in the group of compactly supported Hamiltonian diffeomorphisms of $\mathbb{D}^{2}(c)\times \mathbb{D}^{2}(a)$ contains a $C^{0}$-open subset.
Finally, we show that the aforementioned group equipped with the Hofer distance admits an isometric embedding of an infinite dimensional flat space for suitable values of parameters $c$ and $a$. 
\end{abstract}

\tableofcontents
\pagenumbering{arabic}

\section{Introduction}
\label{introduction section}

A symplectic manifold is an even dimensional smooth manifold that admits a closed non-degenerate $2$-form which is called a symplectic form. A Hamiltonian diffeomorphism of a symplectic manifold $(M,\omega)$ is a diffeomorphism that is the time-one map of the flow of a time dependent vector field $X_{H}$ where $H$ is a smooth compactly supported time-dependent Hamiltonian function on $M$. The vector field $X_{H}$ is determined by the Hamiltonian $H$ by the equation $X_{H} \lrcorner\ \omega = -dH$. \par
If $(M,\omega)$ is a compact symplectic manifold, then $\mathrm{Ham}(M,\omega)$ is the set of all Hamiltonian diffeomorphisms of $M$ that are compactly supported in the interior of $M$. The set $\mathrm{Ham}(M,\omega)$ is a normal subgroup of $\mathrm{Symp}(M,\omega)$ where $\mathrm{Symp}(M,\omega)$ is the set of all diffeomorphisms of $M$ that preserve the symplectic structure $\omega$. \par
This remarkable group (of Hamiltonian diffeomorphisms) has been extensively studied from different points of view, in particular, its geometry and algebraic structure. Other than the natural topologies that one can imagine on this group, such as $C^{0}, C^{1}, C^{\infty}$-topologies, one could also consider natural topologies coming from Finsler structures. The tangent space of $\mathrm{Ham}(M,\omega)$ at identity is $\mathcal{A} := C_{c}^{\infty}(M)$ and it is $C^{\infty}_{0}(M)$ (the space of mean zero smooth functions) when $M$ is a closed manifold. A norm on it defines a Finsler structure on $\mathrm{Ham}(M,\omega)$ and consequently a pseudo-distance between the points of the group. It was proved by Eliashberg-Polterovich \cite{EP93} that the norm $L_{p}$ for all finite $p\geq 1$ defined by
\[
\norm{H}_{L_{p}}:= (\int_{M}|H|^{p}\omega^{n})^{\frac{1}{p}}
\]
defines a degenerate, indeed the zero, pseudo-distance for all finite $p$. But it turns out that the norm $L_{\infty}$ defined by,
\[
\norm{H}_{L_{\infty}}:= \max H - \min H
\]
will result to a non-degenerate pseudo-distance. This highly non-trivial fact was first proved by Hofer \cite{HH90} for $M = \mathbb{R}^{2n}$, see also an alternative proof by Viterbo \cite{Vit92}, then it was extended by Polterovich \cite{LP93} to a wide class of symplectic manifolds with a nice behaviour at infinity, in particular for all closed symplectic manifolds with $[\omega]\in H^{2}(M,\mathbb{Q})$, and finally Lalonde-McDuff \cite{LM95} proved it in full generality using the theory of pseudo-holomorphic curves of Gromov. This metric is called the Hofer metric. 

\subsection{Main theorem}
The interaction of the two topologies, the $C^{0}$-topology and the Hofer topology induced by the Hofer metric, on the group $\mathrm{Ham}(M,\omega)$ is very subtle and has become interesting due to its applications. Recently, Cristofaro Gardiner, Humili\'ere and Seyfaddini \cite{GHS20} presented the first proof of the simplicity conjecture \cite[Problem 42]{MS2} using the PFH spectral invariants. Along the proof, they prove a key lemma which shows an interesting interaction of the $C^{0}$-topology and the Hofer topology.\par
The lemma states that, for a given $\epsilon > 0$ and a disk $B \subset S^{2}$, any $C^{0}$-small enough Hamiltonian diffeomorphism of $S^{2}$ supported in the upper hemisphere  is $\epsilon$-close to a Hamiltonian diffeomorphism supported in $B$ with respect to the Hofer metric, see \cite[Lemma 4.6]{GHS20}. The proof of the lemma, boils down to the symplectic extension and fragmentation results of Entov-Polterovich-Py \cite[Section 6]{EPP12} in dimension $2$. The extension lemmas are very technical and specific to dimension $2$. In \S\ref{extension section} we will use the theory of pseudo-holomorphic curves of Gromov to prove an analogous $4$-dimensional Hamiltonian extension lemma. In \S\ref{fragmentation section}, we prove a fragmentation lemma for the $4$-dimensional symplectic manifold $\mathbb{D}^{2}\times \mathbb{D}^{2}$. These lemmas will be used to prove the following Hofer approximation result for $\mathbb{D}^{2}\times \mathbb{D}^{2}$ in \S\ref{hofer approximation section}: In the following, by $\mathrm{Ham}_{Y}(X)$ we mean, the group of Hamiltonian diffeomorphisms of $X$ compactly supported in the interior of $Y$.

\begin{theorem}[Main theorem]
\label{main theorem intro}
Let $X := S^{2}(a) \times S^{2}(b)$ and $M := \mathbb{D}^{2}(\frac{a}{2}) \times \mathbb{D}^{2}(\frac{b}{2})$. Let $B$ be a topological-disk in $S^{2}(b) \backslash \{pt\}$. Then, for every $\epsilon > 0$ there exists $\delta > 0$ so that the following holds; for every $g\in \mathrm{Ham}_{M}(X)$ satisfying $d_{C^{0}}(g,id) < \delta$ there exist $\psi \in \mathrm{Ham}_{S^{2}(a) \times B}(X)$ satisfying
\[
d_{H}(g, \psi) < \epsilon.
\]
\end{theorem}

\subsection{Outline of the proof}
To prove our Hamiltonian extension lemma for $M = \mathbb{D}^{2}(\frac{a}{2}) \times \mathbb{D}^{2}(\frac{b}{2})$ we use pseudo-holomorphic theory of Gromov. Namely, consider $X = S^{2}(a) \times S^{2}(b)$ and $M$ as a subdomain in $X$, where disks are identified with the upper hemisphere of the corresponding spheres. Let $D_{1}\subset D_{2} \subset D_{3} \subset \mathbb{D}^{2}(\frac{b}{2})$ be the horizontal strips in $\mathbb{D}^{2}(\frac{b}{2})$ as in Lemma \ref{extension lemma 2}. Let $g \in \mathrm{Ham}_{M}(X)$ be a Hamiltonian diffeomorphism that is $C^{0}$-close enough to identity. Then, we would like to find a Hamiltonian diffeomorphism $\psi$ such that coincides with $g$ on $\mathbb{D}^{2}(\frac{a}{2}) \times D_{1}$ and it is supported in $\mathbb{D}^{2}(\frac{a}{2}) \times D_{3}$. Restrict $g$ to $S^{2}(a) \times D(-\epsilon_{2}, 1)$, see \S\ref{extension section}, and extend it by identity to $S^{-}_{\epsilon}(\frac{a}{2}) \times S^{2}(b) \cup S^{2}(a) \times \mathcal{U}^{u}$ where $S^{-}_{\epsilon}(\frac{a}{2})$ is a small enough neighborhood of the lower hemisphere and $\mathcal{U}^{u}$ is a carefully chosen large disk in $S^{2}(b)$, see Figure \ref{extension domains figure}. Then, use Theorem \ref{jhol theorem 2} to extend it to an element $\psi_{u}$ of $\mathrm{Ham}(S^{2}(a) \times S^{2}(b))$. Construct another extension $\psi_{d}$ of the restriction of $g$ to $S^{2}(a) \times D(-1,\epsilon_{2})$, see \S\ref{extension section}, where this time we consider another open subset $\mathcal{U}^{d}$, see Figure \ref{extension domains figure}. Then, the diffeomorphism defined by $\psi:= \psi_{u}\circ \psi_{d}\circ g^{-1}$ restricted to $\mathbb{D}^{2}(\frac{a}{2}) \times \mathbb{D}^{2}(\frac{b}{2})$ will have the desired properties.\par
Let us now sketch the proof of our fragmentation result. Divide $\mathbb{D}^{2}(\frac{b}{2})$ into $N$ horizontal strips. Consider the covering of $\mathbb{D}^{2}(\frac{a}{2}) \times \mathbb{D}^{2}(\frac{b}{2})$ by $\mathbb{D}^{2}(\frac{a}{2}) \times D_{i}$, $i = 1,\dots, N$. Around each intersection line $D_{i} \cap D_{i+1}$, consider very thin horizontal strips $D_{i,1} \subset D_{i,2} \subset D_{i,3}$ and execute the extension lemma on each of these sets of strips and this will fragment $g$ into $g_{1}\circ \dots \circ g_{N}\circ \theta$ where $g_{i}$ is supported in $\mathbb{D}^{2}(\frac{a}{2}) \times D_{i}$ for $i = 1,\dots, N$ and $\theta$ is supported in $\mathbb{D}^{2}(\frac{a}{2}) \times V$ where $V$ is a disjoint union of arbitrary small disks.\par

To prove the main theorem, Theorem \ref{main theorem intro}, first consider large enough integers $k,N > 0$ and the covering $\mathbb{D}^{2}(\frac{a}{2}) \times \underset{i=1}{\overset{kN}\sqcup} D_{i}$ consist of small stabilized horizontal strips. Choose $\delta > 0$ small enough so that any $g\in B_{C^{0}}(id, \delta)$ can be fragmented into $g_{1}\circ \dots\circ g_{kN}\circ \theta$ where $g_{i}$ is supported in $\mathbb{D}^{2}(\frac{a}{2}) \times D_{i}$ and $\theta$ is supported in a disjoint union of stabilized disks with small enough total area. Since the supports of $g_{i}$'s are disjoint, they commute. We now partition them into $N$ groups of cardinality $k$ and denote the composition of the elements of the $i$th group by $f_{i}$. Hence, we have $g = f_{1}\circ \dots \circ f_{N}\circ \theta$. One can find Hamiltonian diffeomorphisms $h_{i}$, $i = 1,\dots, N$ and $h_{\theta}$ of $S^{2}(a) \times S^{2}(b)$ with small Hofer norm that map the support of $f_{i}$, $i=1,\dots,N$ and $\theta$ into $S^{2}(a) \times B$ respectively. Then, the Hamiltonian diffeomorphism $\psi:= \underset{i=1}{\overset{N}\Pi}h_{i}f_{i}h_{i}^{-1}\circ h_{\theta}\theta h_{\theta}^{-1}$ will be Hofer close to $g$ and supported in $S^{2}(a) \times B$.

\subsection{\texorpdfstring{$C^{0}$-continuity of Lagrangian spectral estimators}{Lg}}
Recently, Polterovich and Shelukhin \cite{PS23} showed that a certain family of Lagrangian tori in $M_{a}:= S^{2}(1)\times S^{2}(a)$ is non-displaceable, where $0 < a < 1$. Associated to this non-displaceable family of Lagrangian tori they, in particular, constructed a new functional on the space of time-dependent Hamiltonian functions on $M_{a}$ ($a$ rational) called Lagrangian spectral estimators. These functionals satisfy a long list of remarkable properties, see Theorem \ref{spectral estimators theorem}. Using these spectral estimators they proved many interesting results including the existence of an infinite dimensional flat space in the group $\mathrm{Ham}(S^{2})$ and presented an alternative proof of the Simplicity conjecture. The value of these spectral estimators only depend on the homotopy class of the flow of a mean-zero Hamiltonian. Therefore, they define a functional on the universal cover of $\mathrm{Ham}(S^{2}(1) \times S^{2}(a))$. It is a theorem of Gromov that the group $\mathrm{Ham}_{c}(\mathbb{D}^{2}(\frac{1}{2}) \times \mathbb{D}^{2}(\frac{a}{2}))$ is weakly contractible and in particular it has trivial fundamental group. Hence, restricting the spectral estimators to the subspace $\mathrm{Ham}_{c}(\mathbb{D}^{2}(\frac{1}{2}) \times \mathbb{D}^{2}(\frac{a}{2}))$ we derive a well-defined functional on the Hamiltonian group $\mathrm{Ham}_{c}(\mathbb{D}^{2}(\frac{1}{2}) \times \mathbb{D}^{2}(\frac{a}{2}))$ satisfying many remarkable properties.\par
These functionals denoted by $c_{k,B}$ where $k$ is a positive integer and $B > 0$ is a positive rational number, will not be $C^{0}$-continuous. However, as an application of our main theorem, we show in \S\ref{C^{0}-continuity section} that their difference $\tau_{k,k',B,B'}:= c_{k,B} - c_{k',B'}$ for any two different data $k,B$ and $k',B'$ is a uniformly $C^{0}$-continuous functional.

\begin{theorem}[\texorpdfstring{$C^{0}$}{Lg}-continuity]
\label{C^{0}-continuity of tau intro}
The functional $\tau_{k,k',B,B'}: \mathrm{Ham}_{c}(\mathbb{D}^{2}(\frac{1}{2})\times \mathbb{D}^{2}(\frac{a}{2})) \rightarrow \mathbb{R}$ is a uniformly $C^{0}$-continuous functional, for small enough rational number $0 < a < 1$.
\end{theorem}
\noindent
The area $\frac{1}{2}$ for the disk in the first factor is irrelevant and could be any number in $(0,1]$. 

\subsection{Applications}
In the last two sections of the paper we show some applications of Theorem \ref{C^{0}-continuity of tau intro}. In \S\ref{C^{0}-open in complement of hofer ball} we show an application to an interesting question initially posed by Le Roux \cite{LR10} which concerns the interaction of the $C^{0}$-topology and Hofer topology:
\begin{question}
\label{le roux question}
Let $(M,\omega)$ be a symplectic manifold and let $\mathrm{Ham}(M,\omega)$ be the group of compactly supported Hamiltonian diffeomorphisms of $M$. Let $A > 0$ be a fixed positive number and $d_{H}$ be the Hofer metric, see Definition \ref{hofer norm}. Define the following subset of $\mathrm{Ham}(M,\omega)$,
\[
E_{A}(M,\omega):= \{\phi \in \mathrm{Ham}(M,\omega) : d_{H}(\phi, id) > A\}.
\]
Does the set $E_{A}(M,\omega)$ have a non-empty $C^{0}$ interior $?$
\end{question}

For symplectically aspherical manifolds with infinite spectral diameter, Buhovski, Humili\'ere and Seyfaddini \cite{BHS21} proved that the set $E_{A}(M,\omega)$ contain a non-empty $C^{0}$-interior. In \cite[Theorem 8]{YK22} Y. Kawamoto in particular constructed a $C^{0}$-continuous homogenuous quasimorphism on the Hamiltonian group of $S^{2}(1)\times S^{2}(1)$ and used them to give a positive answer to the above question in this case. For the product symplectic manifold $(M\times M, \omega \oplus -\omega)$ where $(M,\omega)$ is a closed symplectically aspherical manifold Mailhot \cite{PA22} positively answered Le Roux's question. \par
In \S\ref{C^{0}-open in complement of hofer ball} we give a positive answer to Le Roux's question for the $4$-dimensional symplectic manifold $\mathbb{D}^{2}(c) \times \mathbb{D}^{2}(\frac{a}{2})$ where $0 < a < 1$ is any rational number and $0 < c < 1$ is any positive number. As another application of our results, following \cite{PS23}, we will prove in \S\ref{big flat in Ham} that an infinite dimensional flat space isometrically embeds into the group of compactly supported Hamiltonian diffeomorphisms of $\mathbb{D}^{2}(c) \times \mathbb{D}^{2}(\frac{a}{2})$.

\begin{theorem}
\label{big flat in ham intro}
\sloppy The space $C^{\infty}_{c}(0,b)$ isometrically embeds into $\mathrm{Ham}_{c}(\mathbb{D}^{2}(b') \times \mathbb{D}^{2}(a'))$ where $0 < a' < 1$ is any number that satisfies $b < \frac{1}{6}(1 - a')$ and $b'$ is any positive number satisfying $\frac{1}{2} + b < b' < 1$. Here,  we consider the $C^{0}$-distance on the source and the Hofer distance on the target.
\end{theorem}

\vspace{0.75cm}

\begin{center}
\textbf{Acknowledgement}
\end{center}

This research is part of my PhD program at the Université de Montr\'eal under the supervision of Egor Shelukhin. I would like to thank him for proposing the project and guiding me through it. I also thank him for many useful discussions and pointing out the possible applications of our main theorem. I am grateful to Marcelo S. Atallah, Filip Bro\'ci\'c, Dylan Cant for helpful discussions, and Pierre-Alexander Mailhot for creating the pictures. This research was partially supported by Fondation Courtois.

\section{Theory of pseudo-holomorphic curves}
In this section we assert a variant of a theorem from \cite{jhol} which we use in \S\ref{extension section}, see \cite[Theorem 9.4.7]{jhol} for the original version. 
\begin{theorem}
\label{jhol theorem 2}
Let $(M,\omega)$ be a compact connected symplectic $4$-manifold that does not contain any symplectically embedded $2$-sphere with self-intersection number $-1$. Let $A,B \in H_{2}(M,\mathbb{Z})$ be two integer homology classes that are represented by symplectically embedded $2$-spheres and satisfy the following:
$$A.B = 1, \ A.A = 0, \ B.B = 0.$$
Let $\sigma\in \Omega^{2}(S^{2})$ be an area form with $\int_{S^{2}}\sigma = 1$. Then the following holds:
\begin{enumerate}
\item   There is a diffeomorphism $\psi: S^{2}\times S^{2}\rightarrow M$ such that, 
$$\psi^{*}\omega = a\pi_{1}^{*}\sigma + b\pi_{2}^{*}\sigma,\ \ a = \int_{A}\omega,\ b = \int_{B}\omega$$

\item If $U, V\subset S^{2}$ are two open disks and $\iota: U \times S^{2} \cup S^{2}\times V \rightarrow M$ is an embedding such that: 
$$\iota^{*}\omega = a\pi_{1}^{*}\sigma + b\pi_{2}^{*}\sigma,\ \ a = \int_{A}\omega,\ b = \int_{B}\omega$$
$$\iota_{*}([S^{2}\times \{w\}]) = A, \ \ \iota_{*}([\{z\}\times S^{2}])= B$$
for all $z\in U, w \in V$, then for any closed subsets $D \subset U$ and $C \subset V$ the diffeomorphism $\psi$ in $(1)$ can be chosen to agree with $\iota$ on $D \times S^{2} \cup S^{2}\times C$.
\end{enumerate}
\end{theorem}

\begin{proof}
The proof is quite similar to the proof of \cite[Theorem 9.4.7]{jhol}. We may assume that the symplectic embedding $\iota$ can be smoothly extended to a small neighborhood of the closure of $U \times S^{2} \cup S^{2} \times V$ by shrinking $U,V$. Choose a regular compatible almost complex structure on $M$ that pulls back to the standard complex structure by $\iota$ on the closure of $\iota(U \times S^{2} \cup S^{2} \times V)$ and admits no non-constant $J$-holomorphic sphere with non-positive Chern number, use \cite[Remark 3.2.3]{jhol}. Define the Gromov's map $\psi$ using the almost complex structure $J$ and the two $J$-holomorphic spheres $\iota(z_{*}, .)$ and $\iota(., w_{*})$ for some $z_{*} \in U$ and $w_{*} \in V$, see proof of part $(1)$ of \cite[Theorem 9.4.7]{jhol}. The map $\psi$ coincides with $\iota$ on $U \times V$,  $\{z_{*}\} \times S^{2}$ and $S^{2} \times \{w_{*}\}$. Moreover the standard spheres $\{z\} \times S^{2}$ and $S^{2} \times \{w\}$ are $\psi^{*}\omega$-symplectic. Altering $\psi$ on $U \times S^{2}\backslash V \sqcup S^{2} \backslash U \times V$ separately using the following maps as in \cite[Theorem 9.4.7]{jhol} we get the final desired map.
$$\rho(V) \subset V, \ \ \rho(C) = w_{*},\ \ \rho = id\ \mathrm{on}\ S^{2}\backslash V, \ \ \det(d\rho) \geq 0,$$
$$\rho(U) \subset U, \ \ \rho(D) = z_{*},\ \ \rho = id\ \mathrm{on}\ S^{2}\backslash U, \ \ \det(d\rho) \geq 0.$$
\end{proof}

\section{Hamiltonian extension in dimension four}
\label{extension section}

In the following, for any subset $B\subset M$, by $\mathrm{Ham}_{B}(M)$ we shall mean the group of Hamiltonian diffeomorphisms of $M$ that are compactly supported inside $B$.
\begin{definition}[Topological disk]
Let $\mathbb{D}^{2}$ be the standard unit disk in $\mathbb{R}^{2}$. A topological disk in a $2$-dimensional manifold $\Sigma$ is the image of the standard disk under a continuous map $\mathbb{D}^{2} \rightarrow \Sigma$ that is homeomorphic to its image.
\end{definition}

In the following lemma by $\mathbb{D}^{2}(c)$ we mean the standard unit disk in $\mathbb{R}^{2n}$ equipped with the area form $\frac{c}{\pi}\omega_{0}$ where $\omega_{0}$ is the standard area form. In particular when we write $\mathbb{D}^{2}(\frac{c}{2})$ we think of the disk as the upper hemisphere of a sphere $S^{2}(c)$ equipped with area form $c\sigma$ where $\sigma$ is the standard area form on sphere with total area $1$. We also define $D(a,b)$ as follows:
$$D(a, b) := \{(x,y)\in \mathbb{D}^{2} : a \leq y \leq b \}.$$
Before stating the lemma below we fix some data. Let $0 < \epsilon_{1} < \epsilon_{2} < \epsilon_{3} < 1$ be some numbers and $a,b > 0$ be two positive numbers. Denote the set $D(-\epsilon_{i}, \epsilon_{i})$ by $D_{i}$ for $i = 1,2,3$.

\begin{definition}
\label{def:c0_smallness}
Let  $0 < \epsilon_{1} < \epsilon_{2} < c < \epsilon_{3} < 1$ be some numbers and $a,b$ be two positive numbers. Let $M = \mathbb{D}^{2}(\frac{a}{2}) \times \mathbb{D}^{2}(\frac{b}{2})$ and $D_{i}, i=1,2,3$ as above. A diffeomorphism $g$ of $M$ is called $(\epsilon_{1}, \epsilon_{2}, \epsilon_{3}, c)$-small if it satisfies the following conditions:
$$g(\mathbb{D}^{2}(\frac{a}{2}) \times D_{1}) \cup g^{-1}(\mathbb{D}^{2}(\frac{a}{2}) \times D_{1}) \subset \mathbb{D}^{2}(\frac{a}{2}) \times D_{2},$$
$$g(\mathbb{D}^{2}(\frac{a}{2}) \times D_{2}) \cup g^{-1}(\mathbb{D}^{2}(\frac{a}{2}) \times D_{2}) \subset \mathbb{D}^{2}(\frac{a}{2}) \times D(-a,a),$$
$$g^{-1}(\mathbb{D}^{2}(\frac{a}{2}) \times D(\epsilon_{3},1)) \subset \mathbb{D}^{2}(\frac{a}{2}) \times D(a,1),$$
$$g^{-1}(\mathbb{D}^{2}(\frac{a}{2}) \times D(-1, -\epsilon_{3})) \subset \mathbb{D}^{2}(\frac{a}{2}) \times D(-1, -a).$$
Note that being $C^{0}$ close enough to identity implies $(\epsilon_{1}, \epsilon_{2}, \epsilon_{3}, c)$-smallness.
\end{definition}

\begin{lemma}[Extension lemma]
\label{extension lemma 2}
\sloppy 
Let $a,b, \epsilon_{1}, \epsilon_{2}, \epsilon_{3}, c$ and $D_{i}, i=1,2,3$ be as in Definition \ref{def:c0_smallness}. Let $X = S^{2}(a) \times S^{2}(b)$ and $M = \mathbb{D}^{2}(\frac{a}{2}) \times \mathbb{D}^{2}(\frac{b}{2})$. Let $g$ be a Hamiltonian diffeomorphism of $X$ compactly supported in $M$ that is $C^{0}$-small enough. Then, there exist a Hamiltonian diffeomorphism $\psi \in \mathrm{Ham}(X)$ compactly supported in $\mathbb{D}^{2}(\frac{a}{2}) \times D_{3}$ that restricts to $g$ on $\mathbb{D}^{2}(\frac{a}{2}) \times D_{1}$.
\end{lemma}

\begin{proof}
identifying $\mathbb{D}^{2}(\frac{b}{2})$ with the upper hemisphere of $S^{2}(b)$, consider the following subsets of $S^{2}(b)$,
\[
\mathcal{U}^{u} := S^{-}_{\epsilon}(\frac{b}{2}) \cup D(-1,-a) \cup D(-\epsilon_{2},1)
\]
\[
\mathcal{U}^{d} :=  S^{-}_{\epsilon}(\frac{b}{2}) \cup D(a,1) \cup D(-1,\epsilon_{2})
\]
where $S^{-}_{\epsilon}(\frac{b}{2})$ is a small enough neighborhood of the lower hemisphere so that $g$ is identity on $S^{2}(a) \times S^{-}_{\epsilon}(\frac{b}{2})$. See Figure \ref{extension domains figure}.
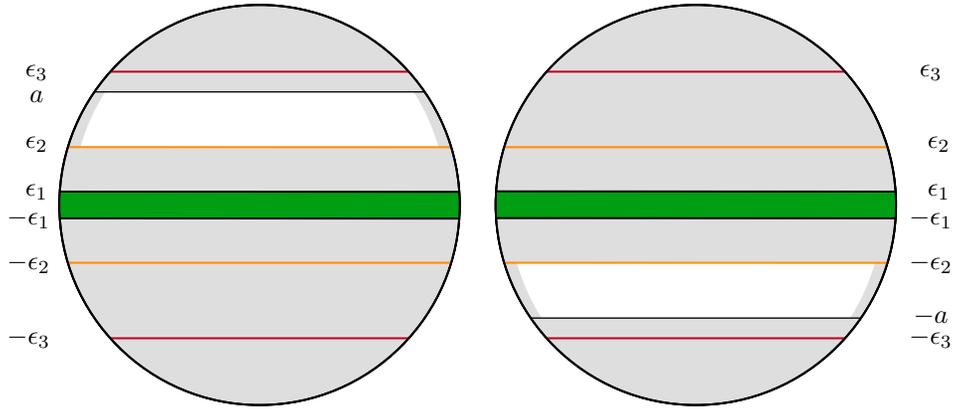
\begin{figure}[ht]
\labellist
\small\hair 2pt
\pinlabel $\epsilon_{3}$ at 10 137
\pinlabel $a$ at 10 128
\pinlabel $\epsilon_{2}$ at 10 110
\pinlabel $\epsilon_{1}$ at 10 92
\pinlabel $-\epsilon_{1}$ at 7 82
\pinlabel $-\epsilon_{2}$ at 7 65
\pinlabel $-\epsilon_{3}$ at 7 37

\pinlabel $\epsilon_{1}$ at 350 92
\pinlabel $-\epsilon_{1}$ at 347 82
\pinlabel $\epsilon_{2}$ at 350 110
\pinlabel $-\epsilon_{2}$ at 347 65
\pinlabel $\epsilon_{3}$ at 347 137
\pinlabel $-a$ at 347 45
\pinlabel $-\epsilon_{3}$ at 347 37
\endlabellist
\centering
\includegraphics{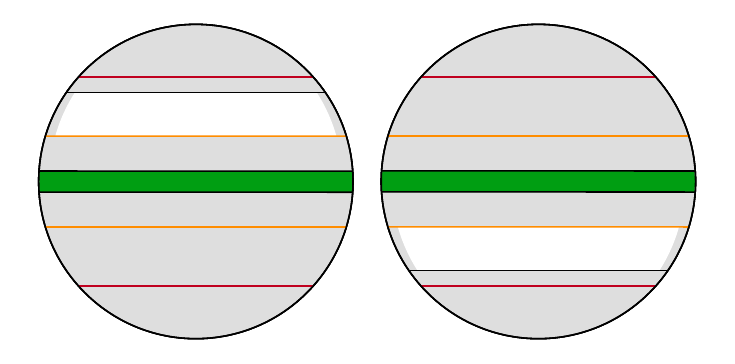}
\caption{Extension domains: $\mathcal{U}^{d}$ is on the left and $\mathcal{U}^{u}$ is on the right.}
\label{extension domains figure}
\end{figure}

\noindent
Restrict ${g}$ to $S^{2}(a) \times D(-\epsilon_{2},1)$ and extend the restriction to a symplectic embedding of $ S^{-}_{\epsilon}(\frac{a}{2}) \times S^{2}(b) \cup S^{2}(a) \times \mathcal{U}^{u}$ into $S^{2}(a) \times S^{2}(b)$ by identity where $S^{-}_{\epsilon}(\frac{a}{2})$ is a small enough neighborhood of the lower hemisphere in $S^{2}(a)$. Use Theorem \ref{jhol theorem 2} to extend the resulting embedding to a symplectomorphism $\psi_{u}: S^{2}\times S^{2}\rightarrow S^{2}\times S^{2}$. Restrict ${g}$ to $S^{2}(a) \times D(-1,\epsilon_{2})$ and extend it by identity to $S^{-}_{\epsilon}(\frac{a}{2}) \times S^{2}(b) \cup S^{2}(a) \times \mathcal{U}^{d}$, and do the same construction to obtain a symplectomorphism $\psi_{d}$ of $S^{2}\times S^{2}$. The resulting symplectomorphisms are Hamiltonian diffeomorphisms. This is implied by Remark \ref{rem:pi0_symp}, the facts that the maps $\psi_{u}, \psi_{d}$ induce identity on the second Homology and $S^{2} \times S^{2}$ is simply connected. The Hamiltonian diffeomorphism $\psi := \psi_{u}\circ \psi_{d}\circ g^{-1}$ satisfies the desired properties. Namely, if $p \in \mathbb{D}^{2}(\frac{a}{2}) \times D(\epsilon_{3},1)$ then since $g^{-1}(p)\in \mathbb{D}^{2}(\frac{a}{2}) \times D(a,1)$ we have
\[
\psi(p) = \psi_{u}\circ\psi_{d}\circ g^{-1}(p) = \psi_{u}\circ g^{-1}(p) = g\circ g^{-1}(p) = p,
\]
if $p\in \mathbb{D}^{2}(\frac{a}{2}) \times D(-1,-\epsilon_{3})$ then
\[
\psi(p) = \psi_{u}\circ\psi_{d}\circ g^{-1}(p) = \psi_{u}\circ g \circ g^{-1}(p) = \psi_{u}(p) = p
\]
and finally if $p \in \mathbb{D}^{2}(\frac{a}{2}) \times D(-\epsilon_{1},\epsilon_{1})$, then
\[
\psi(p) = \psi_{u}\circ \psi_{d}\circ g^{-1}(p) = \psi_{u}\circ g \circ g^{-1}(p) = \psi_{u}(p) = g(p).
\]
Clearly $\psi$ is identity on $S^{-}_{\epsilon}(\frac{a}{2}) \times S^{2}(b) \cup S^{2}(a) \times S^{-}_{\epsilon}(\frac{b}{2})$. Therefore $\psi$  is the desired Hamiltonian diffeomorphism.
\end{proof}

\begin{remark}
\label{rem:pi0_symp}
\sloppy Let $a,b > 0$ be two positive real numbers. Let $G(a,b)$ be the symplectic mapping class group of $S^{2}(a) \times S^{2}(b)$, i.e. $G(a,b) = \pi_{0}(\mathrm{Symp}(S^{2}(a) \times S^{2}(b)))$. In his 1985 paper, Gromov proved that when $a = b$, the group $G(a,b)$ is isomorphic to $\mathbb{Z}_{2}$. Later, McDuff and Abreu \cite{Ab98, Ab00} proved that for all $a \neq b$ the group $G(a,b)$ is a trivial group, i.e. the group $\mathrm{Symp}(S^{2}(a) \times S^{2}(b))$ is connected.
\end{remark}

\begin{remark}
\label{rem:ext_lem_2}
It is expected that an analogous statement should hold for $M = S^{2} \times \mathbb{D}^{2}$ and $M = S^{2}\times S^{2}$. For instance for the second case one should have the following; for annuli $A_{1} \subset A_{2} \subset A_{2} \subset S^{2}$ and a Hamiltonian diffeomorphism $\psi \in \mathrm{Ham}(S^{2}\times S^{2})$ that is $C^{0}$-small enough, there exists $\phi \in \mathrm{Ham}(S^{2}\times S^{2})$ that is supported in $S^{2}\times A_{3}$ and coincides with $\psi$ on $S^{2}\times A_{1}$. One way to approach this, is to prove a refined version of Theorem \ref{jhol theorem 2}. 
\end{remark}

\section{Hamiltonian fragmentation in dimension four}
\label{fragmentation section}
In this section we prove a fragmentation lemma for the symplectic $4$-manifold $\mathbb{D}^{2}(\frac{a}{2}) \times \mathbb{D}^{2}(\frac{b}{2})$. 

\begin{lemma}[Fragmentation lemma]
\label{fragmentation lemma 2}
Let $M = \mathbb{D}^{2}(\frac{a}{2}) \times \mathbb{D}^{2}(\frac{b}{2})$. Let $\rho > 0$ be a positive number and $m > 0$ be a  positive integer. Divide the unit disk $\mathbb{D}^{2}(\frac{b}{2})$ into $m$ horizontal strips with equal area $\frac{b}{2m}$ and denote them by $D_{i}$. Define $U_{i}$ to be the interior of $D_{i}$ for every $i$. Then, there exists $\delta > 0$ such that for every $g \in \mathrm{Ham}_{c}(M)$ with $d_{C^{0}}(g,id) < \delta$ there are $g_{i} \in \mathrm{Ham}_{\mathbb{D}^{2}(\frac{a}{2}) \times U_{i}}(M)$ for $i = 1, \dots, m$, and $\theta \in \mathrm{Ham}_{\mathbb{D}^{2}(\frac{a}{2}) \times U}(M)$ where $U$ is a disjoint union of topological disks in $\mathbb{D}^{2}$ with total area less than $\rho$, so that, $g = g_{1}\circ \dots \circ g_{m}\circ \theta$.
\end{lemma}

Before proving the lemma, we would like to point out that the in \cite[Section 5.3]{EPP12}, the authors suggest that the powerful tool of pseudo-holomorphic theory in dimension four could be used to prove some $C^{0}$-small fragmentation results in dimension four. Although, the fragmentation proved here is not a $C^{0}$-small fragmentation, but it is likely that our method could be used to prove a $C^{0}$-small fragmentation.

\begin{proof}[Proof of Lemma \ref{fragmentation lemma 2}]
For every $i \in \{ 1,\dots, m-1\}$, define $V_{i,1} \subset V_{i,2} \subset V_{i,3} \subset D_{i}\cup D_{i+1}$ be small enough horizontal strips so that $D_{i}\cap D_{i+1} \subset V_{i,1}$ and for all $i \neq j$ we have $V_{i,3}\cap V_{j,3} = \emptyset$, See Figure \ref{fragmentation figure 2}. Let $\delta > 0$ be a small enough positive number that satisfies the following, for every $g\in \mathrm{Ham}(M)$, $d_{C^{0}}(g, id) < \delta$ implies that $g$ satisfies the properties of Lemma \ref{extension lemma 2} for all collections of strips $\{V_{i,j}\}_{j=1}^{3}$. Let $g\in \mathrm{Ham}_{c}(M)$ with $d_{C^{0}}(g,id) < \delta$. Applying Lemma \ref{extension lemma 2} to $g$, we find a Hamiltonian diffeomorphism $\psi_{i}$ of $M$ compactly supported in $\mathbb{D}^{2}(\frac{a}{2}) \times V_{i,3}$ and restricts to $g$ on $\mathbb{D}^{2}(\frac{a}{2}) \times V_{i,1}$. Define $\theta:= \psi_{1}\psi_{2}\dots\psi_{m-1}$. Then we have $g\theta^{-1} = g_{1}\dots g_{m}$ where $g_{i}$ is compactly supported in $\mathbb{D}^{2}(\frac{a}{2}) \times U_{i}$ for all $i = 1,\dots, m$.
\begin{figure}[ht]
\labellist
\small\hair 2pt
\pinlabel $V_{i,2}$ at 70 85
\pinlabel \scalebox{0.7}{$V_{i,1}$} at 261 85
\pinlabel $V_{i,3}$ at 280 85
\endlabellist
\centering
\includegraphics{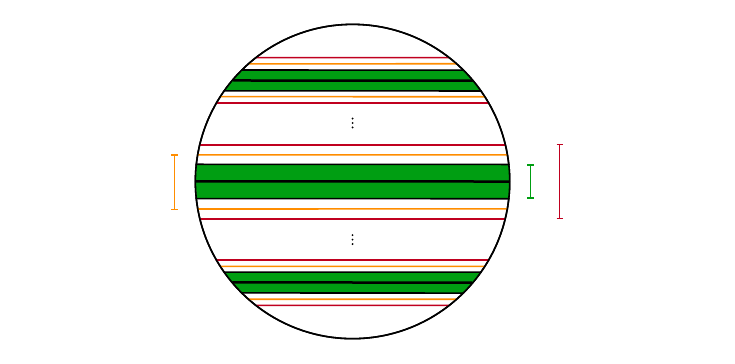}
\caption{Fragmentation pattern.}
\label{fragmentation figure 2}
\end{figure}
\end{proof}

\section{Hofer approximation by fixed small supports}
\label{hofer approximation section}

\begin{definition}[Hofer norm]
\label{hofer norm}
Let $(M,\omega)$ be a symplectic manifold, denote by $C^{\infty}_{c}(M\times [0,1],\mathbb{R})$ the space of compactly supported Hamiltonian functions and let $\mathrm{Ham}(M,\omega)$ be the group of compactly supported Hamiltonian diffeomorphisms. Let $H\in C^{\infty}_{c}(M\times [0,1],\mathbb{R})$, define,
\[
\norm{H}:= \int_{0}^{1}\big(\max{H_{t}} - \min{H_{t}}\big)\ dt.
\]
For a Hamiltonian diffeomorphism $\phi\in \mathrm{Ham}(M,\omega)$ we define the Hofer-norm of $\phi$ by the following:
\[
\norm{\phi}_{H}:= \underset{H\mapsto \phi}\inf \norm{H},
\]
where the infimum is taken over all compactly supported Hamiltonian functions $H$ whose time-one map is $\phi$. For two Hamiltonian diffeomorphisms $\phi, \psi$ we define their Hofer distance by:
\[
d_{H}(\phi, \psi) := \norm{\phi\psi^{-1}}_{H}.
\]
Non-degeneracy of this distance function is highly non-trivial, see introduction.
\end{definition}

\begin{theorem}[Main theorem]
\label{main theorem}
Let $X := S^{2}(a) \times S^{2}(b)$ and $M := \mathbb{D}^{2}(\frac{a}{2}) \times \mathbb{D}^{2}(\frac{b}{2})$. Let $B$ be a topological-disk in $S^{2}(b)$. Then, for every $\epsilon > 0$ there exists $\delta > 0$ so that the following holds; for every $g\in \mathrm{Ham}_{M}(X)$ satisfying $d_{C^{0}}(g,id) < \delta$ there exist $\psi \in \mathrm{Ham}_{S^{2}(a) \times B}(X)$ satisfying
\[
d_{H}(g, \psi) < \epsilon.
\]
\end{theorem}

\begin{remark}
\label{rem:hofer_app_ext}
An analogous statement is expected to hold for $M = S^{2} \times  \mathbb{D}^{2}$ and $M = S^{2} \times S^{2}$. This will be possible by showing a refined version of Theorem \ref{jhol theorem 2}, see Remark \ref{rem:ext_lem_2} as well.
\end{remark}

\begin{remark}
\label{two hofer norms}
In \S\ref{spectral estimators section} we work with an alternative definition of the Hofer norm which is denoted by $\norm{.}_{h}$ and defined as follows:
\[
\norm{\phi}_{h}:= \underset{H\mapsto \phi}\inf \int_{0}^{1}\max{|H_{t}|}\ dt,
\]
where the infimum is taken over all mean zero Hamiltonian functions whose time-one map is $\phi$. Here by mean zero we mean $\int_{M}H_{t}\omega^{n} = 0$ for all $t\in [0,1]$. It is an easy  observation that for every Hamiltonian diffeomorphism $\phi$ we have,
\[
\norm{\phi}_{h}\leq \norm{\phi}_{H}\leq 2\norm{\phi}_{h}.
\]
Namely, it follows from the following two facts; for a function $b\in C^{\infty}([0,1])$ we have $\norm{H + b} = \norm{H}$, for a mean zero Hamiltonian $H$ we have $\max{H_{t}} . \min{H_{t}} \leq 0$ and consequently,
\[
\max{|H_{t}|}\leq \max{H_{t}} - \min{H_{t}}\leq 2\max{|H_{t}|}.
\]
\end{remark}

In the following lemma as usual $\mathbb{D}^{+}(\frac{b}{2})$ denotes the upper hemisphere of $S^{2}(b)$ and $p \in S^{2}(b)$ is the south pole.
\begin{lemma}
\label{displacement lemma}
Let $\rho \in (0,\frac{b}{2})$. Let $B_{1},\dots, B_{k}\subset \mathbb{D}^{+}(\frac{b}{2})$ be a collection of disjoint open topological disks with area less than $\rho$. Let $B \subset S^{2} \backslash \{p\}$ be a topological disk with $area(B) > k\rho$. Then there exists a Hamiltonian diffeomorphism $h \in \mathrm{Ham}(S^{2}(a) \times S^{2}(b))$ supported in $S^{2}(a) \times S^{2}(b) \backslash \{p\}$ such that $h \big(S^{2}(a) \times \cup_{i}B_{i}\big) \subset S^{2}(a) \times B$ and $\norm{h}_{H} \leq 2\rho$, where $\norm{.}_{H}$ is the Hofer norm.
\end{lemma}

\begin{proof}
It is known that there exists an $h_{1} \in \mathrm{Ham}(S^{2}(b))$ supported in $S^{2} \backslash \{p\}$ so that $h_{1}(\cup_{i}B_{i}) \subset B$ and $\norm{h_{1}}_{H} \leq 2\rho$, see \cite[Lemma 4.8]{GHS20}. If $G_{1}: S^{2}(b) \times [0,1]\rightarrow \mathbb{R}$ is a Hamiltonian function generating $h_{1}$ then $G: S^{2}(a) \times S^{2}(b) \times [0,1]\rightarrow \mathbb{R}$ defined by $G(x,y,t):= G_{1}(x,t)$ generates the Hamiltonian diffeomorphism $h = h_{1} \times id$ of $S^{2}(a) \times S^{2}(b)$. We have that, 
$$\int_{0}^{1}\big(\max_{S^{2}(a) \times S^{2}(b)} G(.,.,t) - \min_{S^{2}(a) \times S^{2}(b)} G(.,.,t) \big) dt $$
$$ = \int_{0}^{1} \big( \max_{S^{2}(b)} G_{1}(.,t) - \min_{S^{2}(b)} G_{1}(.,t) \big) dt.$$
Thus, we obtain, $\norm{h}_{H}\leq \norm{h_{1}}_{H}\leq 2\rho$, $h$ is supported in $S^{2}(a) \times S^{2}(b) \backslash \{p_{-}\}$ and $h(S^{2}(a) \times \cup_{i}B_{i}) \subset S^{2}(a) \times B$. 
\end{proof}

We are ready to prove Theorem \ref{main theorem}:
\begin{proof}[proof of Theorem \ref{main theorem}]
Let $N$ be a positive integer so that $\frac{1}{2N} < area(B)$ and $m$ be a multiple of $N$ so that $\frac{2(N+1)}{m} < \epsilon$. Let $\rho = \frac{1}{2m}$. For the choices $\rho$ and $m$ defined above, let $\delta > 0$ be given by Lemma \ref{fragmentation lemma 2}. Let $U_{i}$ be the open strips defined in Lemma \ref{fragmentation lemma 2}. By Lemma \ref{fragmentation lemma 2} there exists $g_{1},\dots, g_{m}$ and $\theta$ in $\mathrm{Ham}(S^{2}(a) \times S^{2}(b))$ such that $g = g_{1}\circ \dots \circ g_{m}\circ \theta$,  $supp(g_{i})\subset \mathbb{D}^{2}(\frac{a}{2}) \times U_{i}$ and $\theta$ is supported in $\mathbb{D}^{2}(\frac{a}{2}) \times V_{N+1}$ where $V_{N+1}$ is disjoint union of topological disks with total area less than $\rho = \frac{1}{2m}$. Since the supports of $g_{i}$'s are disjoint they must commute. Hence, defining
$$f_{j} = \prod_{i \equiv j \ \mathrm{mod}(N)}\ g_{i},\ \ j= 1,\dots, N$$
$$f_{N+1}:= \theta$$
we must have, $g = f_{1}\circ \dots \circ f_{N+1}$. The support of $f_{j}$ for $j\in \{1,\dots, N\}$ lies in $\mathbb{D}^{2}(\frac{a}{2}) \times V_{j}$ where
$$V_{j}:= \bigsqcup_{i \equiv j \ \mathrm{mod}(N)}U_{i}.$$
The area of $V_{j}$ for all $j \in \{1,\dots, N\}$ is $\frac{m}{N}.\frac{1}{2m} = \frac{1}{2N} < \mathrm{area}(B)$ and $\mathrm{area}(V_{N+1}) < \mathrm{area}(B)$ as well. Therefore, by Lemma \ref{displacement lemma} there exist $h_{1},\dots,h_{N+1}\in \mathrm{Ham}(S^{2}(a) \times S^{2}(b))$ supported in $S^{2}(a) \times S^{2}(b) \backslash \{p\}$ such that 
$$h_{j}(S^{2}(a) \times V_{j}) \subset S^{2}(a) \times B,$$
moreover $\norm{h_{j}}_{H}\leq 2.\frac{1}{2m} = \frac{1}{m}$. Define the following Hamiltonian diffeomorphism,
$$\phi:= \prod_{j=1}^{N+1}h_{j}\circ f_{j}\circ h_{j}^{-1}.$$
Then, the support of $\phi$ lies inside the subset $S^{2}(a) \times B$ and we obtain,
$$d_{H}(g,\phi)\leq \sum_{j=1}^{N+1}d_{H}(f_{j},h_{j}f_{j}h_{j}^{-1}) \leq \sum_{j=1}^{N+1} 2\norm{h_{j}}_{H} \leq \frac{2(N+1)}{m} < \epsilon.$$
Note that in the first inequality we have used the bi-invariant property of the Hofer distance. namely, for Hamiltonian diffeomorphisms $a,b,c,d$ we have,
$$d_{H}(ab,cd)\leq d_{H}(ab, cb) + d_{H}(cb, cd) = d_{H}(a,c) + d_{H}(b,d).$$
\end{proof}

\section{Topology of the group $\mathrm{Ham}$}
\label{topology of Ham section}
One of the applications of our main theorem, see Theorem \ref{main theorem}, is proving $C^{0}$-continuity of some spectral invariants defined in \S\ref{spectral estimators section}. These spectral invariants are initially defined on $C^{\infty}(S^{2}(1) \times S^{2}(a))$ for some certain rational numbers $0 < a < 1$. They descend to the group $\widetilde{\mathrm{Ham}}(S^{2}(1) \times S^{2}(a))$. To get an invariant of a Hamiltonian diffeomorphism we need them to descend to the group $\mathrm{Ham}(S^{2}(1) \times S^{2}(a))$. This may not be the case, but it will be the case for $\mathrm{Ham}_{M_{a}}(S^{2}(1) \times S^{2}(a))$ where $M = \mathbb{D}^{2}(\frac{1}{2}) \times \mathbb{D}^{2}(\frac{a}{2})$. This follows from a theorem of Gromov which states that the group $\mathrm{Ham}_{c}(V)$ is contractible, in particular has trivial fundamental group, where $V \subset \mathbb{R}^{4}$ is a compact star-shaped domain, see \cite[Theorem 9.5.2]{jhol}. Therefore, in \S\ref{C^{0}-continuity section} we will have a well-defined invariant on the group $\mathrm{Ham}_{c}(\mathbb{D}^{2}(\frac{1}{2}) \times \mathbb{D}^{2}(\frac{a}{2}))$ which we prove that is uniformly $C^{0}$-continuous.

\begin{remark}
Following the lines of \cite{Eva11}, one can prove that the group $\mathrm{Ham}_{c}(S^{2}(a) \times \mathbb{D}^{2}(c))$ is also contractible, in particular has a trivial fundamental group, hence one has a well-defined invariant on this group by restricting the invariants defined in \S\ref{spectral estimators section}. Hence, by Remark \ref{rem:ext_lem_2} and Remark \ref{rem:hofer_app_ext} it is expected that the invariants defined on $\mathrm{Ham}_{c}(S^{2}(a) \times \mathbb{D}^{2}(c))$ are also uniformly $C^{0}$-continuous.
\end{remark}

\section{Lagrangian spectral estimators}
\label{spectral estimators section}
Let $M_{a} = S^{2}(1) \times S^{2}(a)$ where $0 < a < 1$ is a positive rational number. In \cite{PS23}, Polterovich and Shelukhin constructed new functionals on the space of time dependent Hamiltonian functions (spectral estimators), on the group of Hamiltonian diffeomorphisms (group estimators) and on the Lie algebra of functions on a symplectic manifold (algebra estimators), for the symplectic manifold $M_{a}$, that satisfy a number of remarkable properties. In this section we recall the existence theorem of their functional on the space of time dependent Hamiltonian functions (spectral estimators) and list their properties, see \cite[Section 2]{PS23} for more details. We will define a real-valued invariant $\tau_{k,k',B,B'}$ as the difference of two of the spectral estimators and we prove that its restriction to the Hamiltonian functions that are compactly supported in the subdomain $N_{a}:= \mathbb{D}^{2}(\frac{1}{2}) \times \mathbb{D}^{2}(\frac{a}{2})$ descends to the Hamiltonian group $\mathrm{Ham}_{c}(N_{a})$, where $\mathbb{D}^{2}(\frac{1}{2})$ and $\mathbb{D}^{2}(\frac{a}{2})$ are identified with the upper hemisphere of $S^{2}(1)$ and $S^{2}(a)$. In \S\ref{C^{0}-continuity section} we prove that the map
\[
\tau_{k,k',B,B'}: \mathrm{Ham}_{c}(N_{a})\rightarrow \mathbb{R}
\]
is uniformly $C^{0}$-continuous. 

\subsection{Existence of spectral estimators}
\label{existence of spectral estimators section}
Before we state the existence theorem of spectral estimators we need a few definition and notations. We shall follow the notation in \cite{PS23}. Let $z : S^{2}(1) \rightarrow \mathbb{R}$ be the height function on $S^{2}(1)$, where we think of $S^{2}(1)$ as the standard sphere with radius $\frac{1}{2}$ in $\mathbb{R}^{3}$ with the area form $\frac{1}{\pi}\omega$. Let $0 < C < B$ be two positive numbers and $k > 0$ be an integer such that $2B + (k-1)C = 1$. Denote $l^{0,j}_{k,B} := z^{-1}(-\frac{1}{2} + B + jC)$ for $j = 0,\dots, k-1$ and let $l^{0}_{k,B}$ be the union of the circles $l^{0,j}_{k,B}$. Let $a$ be a positive number that satisfies $0 < \frac{a}{2} < B-C$. Let $S$ be the equator of the factor $S^{2}(a)$ in $M_{a}$. Define the following subsets, 
\[
L_{k,B}:= l^{0}_{k,B}\times S
\]
\[
L^{j}_{k,B}:= l^{0,j}_{k,B}\times S,\ \ j = 0,\dots, k-1.
\]
Let us now recall an important invariant called the Calabi invariant,
\begin{definition}(Calabi invariant)
Let $(M,\omega)$ be an $2n$-dimensional symplectic manifold and $U\subset M$ be a displaceable open subset. The Calabi invariant is a map $Cal : \widetilde{\mathrm{Ham}_{c}(U)}\rightarrow \mathbb{R}$ defined by,
\[
Cal(\{\phi^{t}_{H}\}):= \int_{0}^{1}\int_{U}H_{t}\omega^{n}\ dt,
\]
where $H$ is a Hamiltonian supported in $[0,1]\times U$ and generates the class $[\{\phi^{t}_{H}\}]$. The Calabi invariant is a well-defined homomorphism.
\end{definition}

We are now ready to state the existence theorem of the spectral estimators,

\begin{theorem}\cite[Theorem F]{PS23}
\label{spectral estimators theorem}
Let $k,B,a$ be as above where $B,a$ are rational numbers, and let $M_{a} = S^{2}(1) \times S^{2}(a)$. There exist a map $c_{k,B}: C^{\infty}([0,1]\times M_{a})\rightarrow \mathbb{R}$ that satisfies the following properties:
\begin{itemize}
\item (Hofer-Lipschitz) For each $G,H \in C^{\infty}([0,1]\times M_{a})$ we have, 
\[
|c_{k,B}(G) - c_{k,B}(H)|\leq \int_{0}^{1}\max{|G_{t} - H_{t}|}dt.
\]
\item (Monotonicity) For $G,H \in C^{\infty}([0,1]\times M_{a})$ and $G \leq H$ we have
\[
c_{k,B}(G)\leq c_{k,B}(H).
\]
\item (Normalization) For $G \in C^{\infty}([0,1]\times M_{a})$ and $b \in C^{\infty}([0,1])$ we have
\[
c_{k,B}(G + b) = c_{k,B}(G) + \int_{0}^{1}b(t)\ dt.
\]
\item (Lagrangian control) If $H \in C^{\infty}([0,1]\times M_{a})$ and $H(t,-)_{|_{L^{j}_{k,B}}}\equiv c_{j}(t)\in \mathbb{R}$, then we have
\[
c_{k,B}(H) = \frac{1}{k}\underset{0\leq j< k}\sum \int_{0}^{1}c_{j}(t)\ dt.
\]
\item (Independence of Hamiltonian) The following function is well defined:
\[
\widetilde{\mathrm{Ham}}(M_{a})\rightarrow \mathbb{R}, \ \ \ [\widetilde{\phi}] \mapsto c_{k,B}(H),
\]
where $H$ is a mean zero Hamiltonian that generates the class $[\widetilde{\phi}]$. By mean zero we mean $\int_{M}H_{t}\omega^{n} = 0$ for all $t\in [0,1]$.
\item (Sub-additivity) For all $\widetilde{\phi},\widetilde{\psi} \in \widetilde{\mathrm{Ham}}(M_{a})$ we have,
\[
c_{k,B}(\widetilde{\phi}\widetilde{\psi}) \leq c_{k,B}(\widetilde{\phi}) + c_{k,B}(\widetilde{\psi}).
\]
\item (Calabi property) If $H \in C^{\infty}([0,1]\times M_{a})$ is supported in $[0,1]\times U$ for an open subset $U$ disjoint from $L_{k,B}$, then,
\[
c_{k,B}(\widetilde{\phi_{H}}) = -\frac{1}{vol(M_{a})}Cal(\widetilde{\phi_{H}}).
\]
\item (Controlled additivity) Let $\widetilde{\psi} \in \widetilde{\mathrm{Ham}}(M_{a})$ such that there is a Hamiltonian $H$ that is supported in $[0,1]\times U$ for an open subset $U$ disjoint from $L_{k,B}$, and generates $\widetilde{\psi}$. Then for all $\widetilde{\phi}\in \widetilde{\mathrm{Ham}}(M_{a})$ we have,
\[
c_{k,B}(\widetilde{\psi}\widetilde{\phi}) = c_{k,B}(\widetilde{\psi}) + c_{k,B}(\widetilde{\phi}).
\]
\end{itemize}
\end{theorem}

We shall consider the restriction of the functions $c_{k,B}$ to the Hamiltonian functions that are compactly supported in $N_{a} = \mathbb{D}^{2}(\frac{1}{2}) \times \mathbb{D}^{2}(\frac{a}{2})$ where we identify $\mathbb{D}^{2}(\frac{1}{2})$ and $\mathbb{D}^{2}(\frac{a}{2})$ with the upper hemisphere of $S^{2}(1)$ and $S^{2}(a)$. By the "Independence of Hamiltonian" property of the spectral estimators, to show that they descend to the group $\mathrm{Ham}_{c}(N_{a})$, it is enough to know that the fundamental group of $\mathrm{Ham}_{c}(N_{a})$ is trivial which is the case, see \S\ref{topology of Ham section}. Finally, we have a well-defined map $c_{k,B}: \mathrm{Ham}_{c}(N_{a})\rightarrow \mathbb{R}$ that in particular satisfies the following  properties:

\begin{itemize}
\label{c_{k,B} properties}
\item (Hofer-Lipschitz) For $\phi, \psi\in \mathrm{Ham}_{c}(N_{a})$ we have
\[
|c_{k,B}(\phi) - c_{k,B}(\psi)| \leq \norm{\phi\psi^{-1}}_{h}.
\]
See Remark \ref{two hofer norms} for the definition of $\norm{.}_{h}$ and its relation with the other Hofer-norm defined in Definition \ref{hofer norm}.

\item (Calabi property) For $\phi \in \mathrm{Ham}_{c}(N_{a})$ that is supported in an open subset $U$ disjoint from $L_{k,B}$ we have,
\[
c_{k,B}(\phi) = -\frac{1}{vol(M_{a})}Cal(\widetilde{\phi}),
\]
where $\widetilde{\phi}$ is a lift of $\phi$ inside $\widetilde{\mathrm{Ham}_{c}}(U)$.
\item (Controlled additivity) Let $\psi \in \mathrm{Ham}_{c}(N_{a})$ be supported in an open subset $U$ disjoint from $L_{k,B}$ and let $\phi \in \mathrm{Ham}_{c}(N_{a})$ be any Hamiltonian diffeomorphism, then,
\[
c_{k,B}(\psi\phi) \overset{(*)}= c_{k,B}(\psi) + c_{k,B}(\phi)
\]
To see this, choose a lift $\widetilde{\psi\phi} \in \widetilde{\mathrm{Ham}_{c}}(N_{a})$ of $\psi\phi$ and let $H \in C^{\infty}([0,1]\times M_{a})$ be a mean  zero Hamiltonian supported in $[0,1]\times N_{a}$ that generates the chosen lift. Choose a lift $\widetilde{\psi}$ of $\psi$ in $\widetilde{\mathrm{Ham}_{c}}(U)$ and let $G\in C^{\infty}([0,1]\times M_{a})$ be supported in $[0,1]\times U$, have mean zero and generate the lift. Then, the Hamiltonian function $\overline{G} \# H$ will have mean zero and generates a lift of $\phi$ in $\widetilde{\mathrm{Ham}_{c}}(N_{a})$ which we call $\widetilde{\phi}$. Hence, by the "Controlled additivity property" in Theorem \ref{spectral estimators theorem} and the choices we made, the equality $(*)$ is implied. 
\end{itemize}

\subsection{\texorpdfstring{$C^{0}$}{Lg}-continuity of \texorpdfstring{$\tau_{k,k',B,B'}$}{Lg}}
\label{C^{0}-continuity section} 
Here, we prove the $C^{0}$-continuity of the invariant $\tau_{k,k',B,B'}: \mathrm{Ham}_{c}(N_{a})\rightarrow \mathbb{R}$ defined below. To prove this, we will be using our main theorem that we proved in \S\ref{hofer approximation section}, see Theorem \ref{main theorem}, the Hofer-Lipschitz property of $c_{k,B}$ invariants that was stated in \S\ref{existence of spectral estimators section}, see Properties \ref{c_{k,B} properties}, and the fact that $\tau$ is invariant under some specific perturbations, see Lemma \ref{invariance under perturbation} below.

\begin{definition}
\label{tau invariant}
Let $k,B, C$ and $k',B', C'$ be as above, $0 < \frac{a}{2} < \min\{B - C, B' - C'\}$, and $B,B',a$ be rational numbers. Let $M_{a}, N_{a}$ be defined as before. Then we define $\tau_{k,k',B,B'}: \mathrm{Ham}_{c}(N_{a})\rightarrow \mathbb{R}$ by the following:
\[
\tau_{k,k',B,B'}(\phi):= c_{k,B}(\phi) - c_{k',B'}(\phi).
\]
We will be omitting the indices unless necessary.
\end{definition}

To prove the $C^{0}$-continuity of the invariant $\tau$ in \S\ref{C^{0}-continuity section}, we use a property of the invariant proved in \cite[Theorem L]{PS23}, which for the convenience of the reader we write it in the lemma below,

\begin{lemma}
\label{invariance under perturbation}
Let $\tau = \tau_{k,k',B,B'}: \mathrm{Ham}_{c}(N_{a})\rightarrow \mathbb{R}$ be the invariant defined in Definition \ref{tau invariant}. Let $U\subset N_{a}$ be an open subset that is disjoint from $L_{k,B}$ and $L_{k',B'}$. Let $\phi \in \mathrm{Ham}_{U}(N_{a})$ be a Hamiltonian diffeomorphism of $N_{a}$ compactly supported in $U$ and let $\theta\in \mathrm{Ham}_{c}(N_{a})$ be any Hamiltonian diffeomorphism. Then, $\tau(\theta\phi) = \tau(\theta)$. In particular, $\tau(\phi) = 0$.
\end{lemma}

\begin{proof}
\[
\tau(\theta\phi) = c_{k,B}(\theta\phi) - c_{k',B'}(\theta\phi)
\]
\[
\overset{(1)}= \big(c_{k,B}(\theta) + c_{k,B}(\phi)\big) - \big(c_{k',B'}(\theta) + c_{k',B'}(\phi)\big)
\]
\[
\overset{(2)}= c_{k,B}(\theta) - c_{k',B'}(\theta) = \tau(\theta).
\]
The equality $(1), (2)$ follow from the controlled additivity and Calabi properties of both $c_{k,B}, c_{k',B'}$ respectively. (See the Properties \ref{c_{k,B} properties}.) For the last part it is enough to set $\theta = id$.
\end{proof}

\begin{theorem}($C^{0}$-Continuity)
\label{C^{0}-continuity of tau}
The invariant $\tau_{k,k',B,B'}: \mathrm{Ham}_{c}(N_{a})\rightarrow \mathbb{R}$ is $C^{0}$-continuous.
\end{theorem}

\begin{proof}
Let $\epsilon > 0$ be a given positive number and $B$ be an open topological disk in $\mathbb{D}^{2}(\frac{1}{2})$ that is disjoint from $l^{0}_{k,B}$ and $l^{0}_{k',B'}$, and define $U:= B \times \mathbb{D}^{2}(\frac{a}{2})$. Let $\delta > 0$ be given by Theorem \ref{main theorem} for $\epsilon/2$ and $B$. (Notice that the factors are flipped compare to Theorem \ref{main theorem}.) Let $\theta \in \mathrm{Ham}_{c}(N_{a})$ be any Hamiltonian diffeomorphism. Define the following $C^{0}$-neighborhood of $\theta$,
\[
\mathcal{V}_{\delta}(\theta):= \{\theta\phi : d_{C^{0}}(\phi, id) < \delta\}.
\]
Choose any element $\theta\phi \in \mathcal{V}_{\delta}(\theta)$ and let $\psi\in \mathrm{Ham}_{c}(N_{a})$ be the Hamiltonian supported in $U$ that is given by the Theorem \ref{main theorem} for $\phi$, then we have,
\[
|\tau(\theta \phi) - \tau(\theta)| = |\tau(\theta\phi) - \tau(\theta\psi)|
\]
\[
\leq 2\norm{(\theta\phi)(\theta\psi)^{-1}}_{h} = 2 \norm{\phi\psi^{-1}}_{h}\leq 2 \norm{\phi\psi^{-1}}_{H} \leq 2\frac{\epsilon}{2} = \epsilon.
\]
The first equality is proved in Lemma \ref{invariance under perturbation}, the first inequality is followed by the Hofer-Lipschitz property of both $c_{k,B}, c_{k',B'}$, see Properties \ref{c_{k,B} properties}, the equality afterwards is by the conjugation invariance of the Hofer norm, and for the inequality between two different Hofer norms see Remark \ref{two hofer norms}.
\end{proof}

Note that the same statement holds for $N_{c,a}:= \mathbb{D}^{2}(c) \times \mathbb{D}^{2}(a)$ for any $0 < c,a < 1$ as soon as an invariant $\tau_{k,k',B,B'}$ is well-defined.

\section{Applications}

\subsection{\texorpdfstring{$C^{0}$}{Lg}-open sets in the complement of Hofer balls}
\label{C^{0}-open in complement of hofer ball}
In this section we show a simple application of our $C^{0}$-continuity result to the following question which was initially posed by Le Roux in \cite{LR10}:
\begin{question}
Let $(M,\omega)$ be a symplectic manifold and let $\mathrm{Ham}(M,\omega)$ be the group of compactly supported Hamiltonian diffeomorphisms of $M$. Let $A > 0$ be a fixed positive number and $d_{H}$ be the Hofer metric on the group, see Definition \ref{hofer norm}. Define the following subset of $\mathrm{Ham}(M,\omega)$,
\[
E_{A}(M,\omega):= \{\phi \in \mathrm{Ham}(M,\omega) : d_{H}(\phi, id) > A\}.
\]
Does the set $E_{A}(M,\omega)$ have a non-empty $C^{0}$ interior $?$
\end{question}

Here we consider the symplectic manifold $\mathbb{D}^{2}(\frac{1}{2}) \times \mathbb{D}^{2}(\frac{a}{2})$ where $0 < a < 1$ is any rational number.

\begin{theorem}
\label{c^{0}-open set in complement of hofer ball theorem}
Let $N_{a} = \mathbb{D}^{2}(\frac{1}{2}) \times \mathbb{D}^{2}(\frac{a}{2})$ where $0 < a < 1$ is any rational number. Then the set $E_{A}(N_{a})$ has non-empty $C^{0}$-interior for every $A > 0$.
\end{theorem}

\begin{proof}
Let $A > 0$ be a positive number. Let $B,B' > a$ be some rational numbers and let $C, k, C', k'$ be some positive numbers that satisfy the assumptions of Definition \ref{tau invariant}. Consider the functional $\tau_{k,k',B,B'}$, which was proved to be $C^{0}$-continuous in Theorem \ref{C^{0}-continuity of tau}. Let $\phi \in \mathrm{Ham}(N_{a})$ be a Hamiltonian diffeomorphism that satisfies $|\tau(\phi)| > 2A + 1$. Now consider the $C^{0}$-ball around $\phi$ with radius $\delta > 0$, $B_{C^{0}}(\phi, \delta)$, where $\delta$ is such that, if $d_{C^{0}}(\psi, \phi) < \delta$ then 
\[
|\tau(\psi) - \tau(\phi)| < 1.
\]
Now for every $\psi \in B_{C^{0}}(\phi, \delta)$ we have,
\[
d_{H}(\psi, id) \geq \frac{1}{2}|\tau(\psi)| > \frac{1}{2}(2A) = A.
\]
Therefore we have,
\[
B_{C^{0}}(\phi, \delta)\subset E_{A}(N_{a}).
\]
Here, the area $\frac{1}{2}$ of the disk is irrelevant and it could be any positive number $c \in (0,1]$.
\end{proof}

\subsection{Infinite dimensional flats in the group $\mathrm{Ham}$}
\label{big flat in Ham}
In this section we answer the question of whether one can isometrically embed a flat space into $\mathrm{Ham}_{c}(N_{r,s})$ equipped with the Hofer distance for some certain real parameters $r,s > 0$ where $N_{r,s} = \mathbb{D}^{2}(r) \times \mathbb{D}^{2}(s)$.

\begin{theorem}
\label{big flat in Ham theorem}
\sloppy The space $(C^{\infty}_{c}(0,b), d_{C^{0}})$ isometrically embeds into $(\mathrm{Ham}_{c}(N_{b',a'}), d_{H})$ where $a' < 1$ is any positive number satisfying  $b < \frac{1}{6}(1 - a')$ and $b'$ satisfies $\frac{1}{2} + b < b' < 1$. Here, $d_{C^{0}}$ is the $C^{0}$-distance and $d_{H}$ is the Hofer distance.
\end{theorem}

\begin{proof}
Let $h\in C^{\infty}_{c}(0,b)$ and define $h^{\#}: [-\frac{1}{2}, \frac{1}{2}]\rightarrow \mathbb{R}$ as follows:
\[
h^{\#}(z):= \begin{cases}
h(z) & z\in (0,b) \\
-h(2b - z) & z\in (b, 2b) \\ 
h(-z) & z \in (-b, 0)\\
-h(2b + z) & z \in (-2b, -b)\\
0 & o.w,
\end{cases}
\]
see Figure \ref{twist maps}. Let $z: S^{2}(1)\rightarrow [-\frac{1}{2}, \frac{1}{2}]$ be the normalized moment map of the natural Hamiltonian action of $S^{1}$ on $S^{2}(1)$. (When $S^{2}$ is the standard sphere in $\mathbb{R}^{3}$ with radius $\frac{1}{2}$ and area form $\frac{1}{\pi}\omega$, the map $z$ is just the height function.) Let $0 < a < 1$ be a rational number slightly bigger than $a'$ so that it satisfies $\frac{a}{2} < a' < a$ and $b < \frac{1}{6}(1 - a)$. Let $\beta: S^{2}(a) \rightarrow \mathbb{R}$ be a smooth cut off function which on $\mathbb{D}^{2}(a')$ is radial and is $1$ on $\{r \leq R - \epsilon\}$, is $0$ on $\{r \geq R - \frac{\epsilon}{2}\}$ and vanishes on $S^{2}(a) \backslash \mathbb{D}^{2}(a')$. Here, $R$ is the radius of $\mathbb{D}^{2}(a')$, $\epsilon > 0$ is a small number depending on $a, a'$ so that $\mathbb{D}^{2}(\frac{a}{2}) \subset \{r \leq R - \epsilon\}$. Set $\Gamma(h): S^{2}(1) \times S^{2}(a) \rightarrow \mathbb{R}$ by $(x,y)\mapsto \beta(y) h^{\#}(z(x))$.

Define the following homomorphism:
\[
\Psi : C^{\infty}_{c}(0,b)\rightarrow \mathrm{Ham}_{c}(N_{b',a'})
\]
\[
h\mapsto \phi^{1}_{\Gamma(h)},
\]
where $\phi^{1}_{\Gamma(h)}$ is the restriction of the time-one map of the autonomous Hamiltonian $\Gamma(h)$ of $M_{a}$ to the subdomain $N_{b',a'}$. We will argue as in the proof of \cite[Theorem A]{PS23}. Let $\widetilde{\Psi}: C^{\infty}_{c}(0,b)\rightarrow \widetilde{\mathrm{Ham}_{c}}(N_{b',a'})$ be the lift of $\Psi$ that takes $h$ to the homotopy class $[\{\phi^{t}_{\Gamma(h)}\}]$. Then, we have,
\[
\norm{\widetilde{\Psi}(h)}_{h} \leq \int_{0}^{1}\max|\Gamma(h)|\ dt = \norm{h}_{C^{0}},
\]
where $\norm{.}_{h}$ is defined as the infimum of the Hofer norm where the infimum is taken over all mean zero Hamiltonian functions generating the same homotopy class, see Remark \ref{two hofer norms}. Note that the Hamiltonian $\Gamma(h)$ is mean zero. To finish the proof we shall prove that the reverse inequality also holds. Since, the map $\Psi$ is a homomorphism and the Hofer norm is invariant under the inverse operation, without loss of generality, we assume that $\norm{h}_{C^{0}} = h(x_{0}) > 0$ for some $x_{0}\in (0,b)$. For $a,b$ as in the statement we consider the invariant $c_{2,B_{i}}: \widetilde{\mathrm{Ham}_{c}}(N_{b',a'}) \rightarrow \mathbb{R}$ from Theorem \ref{spectral estimators theorem} where $B_{i} = \frac{1}{2} - x_{i}$ and $\{x_{i}\}_{i\geq 1}$ is an increasing sequence of rational numbers converging to $x_{0}$ and lie in $(0,b)$, see Figure \ref{twist maps}. 

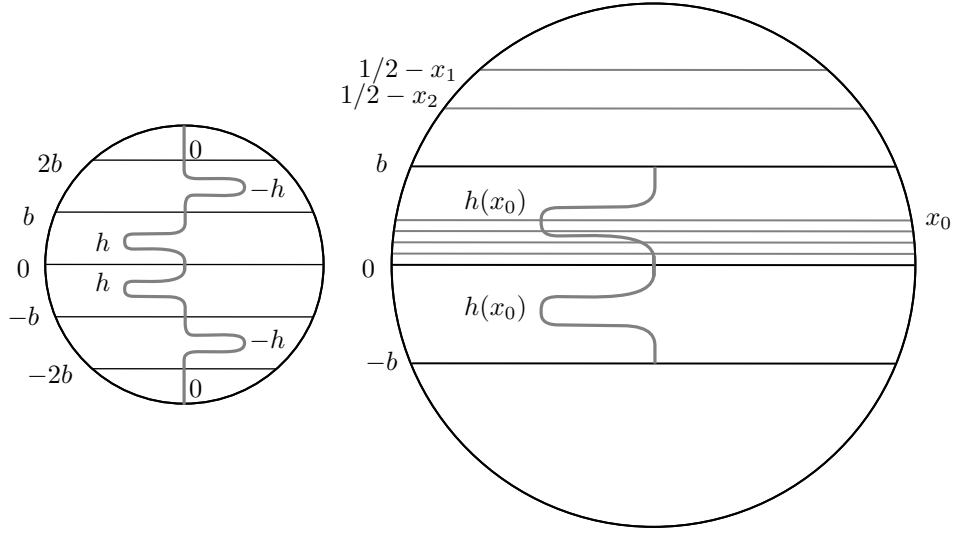
\begin{figure}[ht]
\labellist
\small\hair 2pt
\pinlabel ${1/2 - x_{1}}$ at 150 180
\pinlabel ${1/2 - x_{2}}$ at 143 170
\pinlabel $0$ at 70 150
\pinlabel $h$ at 35 115 
\pinlabel $h$ at 35 100
\pinlabel $-h$ at 97 135 
\pinlabel $0$ at 5 105
\pinlabel $-h$ at 97 78
\pinlabel $0$ at 70 60
\pinlabel $b$ at 140 145
\pinlabel $2b$ at 15 145
\pinlabel $b$ at 7 125
\pinlabel $x_{0}$ at 350 123
\pinlabel $h(x_{0})$ at 183 130
\pinlabel $0$ at 135 105
\pinlabel $h(x_{0})$ at 183 90
\pinlabel $-b$ at 140 70
\pinlabel $-b$ at 5 87
\pinlabel $-2b$ at 15 65
\endlabellist
\centering
\includegraphics{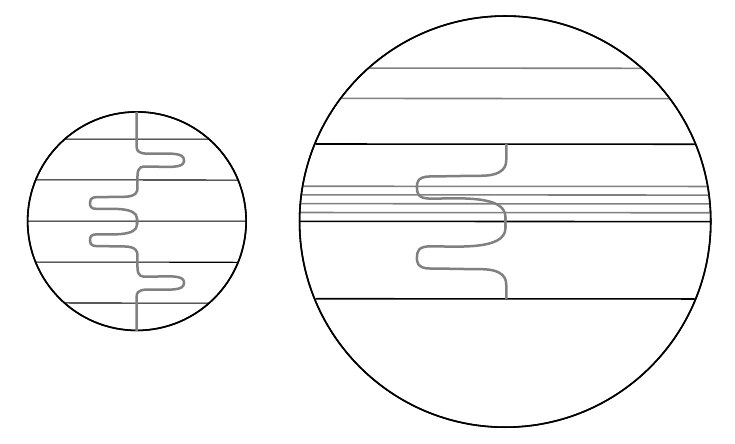}
\caption{$h^{\#}$ is on the left, an increasing sequence of rational levels converging to the level $x_{0}$ shown on the right picture.}
\label{twist maps}
\end{figure}

Note that, in order to have the invariant $c_{2,B_{i}}$ defined, we need the parameters to satisfy the following,
\[
\mathrm{if}\ C_{i} = 1 - 2B_{i} \implies \frac{a}{2} < B_{i} - C_{i},
\]
which hold since we have,
\[
b < \frac{1}{6}(1 - a) \implies \frac{a}{2} < \frac{1}{2} - 3b < \frac{1}{2} - 3x_{i} = B_{i} - C_{i}.
\]
By the Hofer-Lipschitz property of the invariant $c_{2,B_{i}}$ we have,
\[
c_{2,B_{i}}([\{\phi^{t}_{\Gamma(h)}\}]) \leq \norm{\widetilde{\Psi}(h)}_{h},
\]
and by the Lagrangian-control property we have,
\[
c_{2,B_{i}}([\{\phi^{t}_{\Gamma(h)}\}]) = \frac{1}{2}\big(h(\frac{1}{2} - B_{i}) + h(-\frac{1}{2} + B_{i})\big)
\]
\[
=  \frac{1}{2}\big(h(x_{i}) + h(-x_{i})\big) = h(x_{i}).
\]
So, we derive the following inequality:
\[
\norm{\widetilde{\Psi}(h)}_{h} \geq c_{2,B_{i}}([\{\phi^{t}_{\Gamma(h)}\}]) = h(x_{i})\ \ \text{for all}\ i\geq 1 \implies
\]
\[ 
\norm{\widetilde{\Psi}(h)}_{h} \geq \norm{h}_{C^{0}}.
\]
Therefore, the map $\widetilde{\Psi}$ is an isometric embedding, i.e. for all $h\in C^{\infty}_{c}(0,b)$ we have,
\[
\norm{\widetilde{\Psi}(h)}_{h} = \norm{h}_{C^{0}}.
\]
Since the space $\mathrm{Ham}_{c}(N_{b',a'})$ is a weakly contractible space, see \S\ref{topology of Ham section}, in particular it has a trivial fundamental group, hence the map $\Psi$ descends to an isometric embedding of $C^{\infty}_{c}(0,b)$ into $\mathrm{Ham}(N_{b',a'})$.
\end{proof}

\newpage
\bibliographystyle{abbrvnat}
\bibliography{main.bib} 
\nocite{*}

\end{document}